\documentclass{article}
\usepackage{xypic,amscd,amsmath,amsthm,amssymb,textcomp}
\newtheorem{theorem}{Theorem}[subsection]
\newtheorem{lemma}[theorem]{Lemma}

\newtheorem{proposition}[theorem]{Proposition}
\newtheorem{corollary}[theorem]{Corollary}

\newcommand{\Hom}{\mathrm{Hom}}

\newcommand{\PP}{\mathbb{P}}

\newcommand{\RR}{\mathbb R}
\newcommand{\QQ}{\mathbb Q}
\newcommand{\CC}{\mathbb{C}}

\newcommand{\ZZ}{\mathbb{Z}}

\newcommand{\rk}{{\operatorname{rk}}}

\newcommand{\n}{\noindent}
\theoremstyle{definition}
\newtheorem{definition}[theorem]{Definition}

\theoremstyle{remark}
\newtheorem{remark}[theorem]{Remark}

\date{}
\begin{document}
\title{Cubic fourfolds containing a plane and $K3$ surfaces of Picard rank two} 

\author{Federica Galluzzi \\ \thanks{
Federica~Galluzzi 
Dipartimento di Matematica 
Universit\`a di Torino,
Via Carlo Alberto n.10 ,Torino 
10123, ITALIA. 
e-mail: \texttt{federica.galluzzi@unito.it}}}

\maketitle

\begin{abstract}

\n
We present some new examples of families of cubic hypersurfaces in $\PP ^5 (\CC)$ containing a plane whose associated quadric bundle does not have a rational section.  

\smallskip
\n
{\it Mathematics Subject Classification (2010)}:  14E08,14C30,14J28,14F22.
\end{abstract}

\section{Introduction}\label{intro}

Let $X$ be a smooth cubic hypersurface in $\PP^5(\CC) .\,$ Investigating the rationality of $X$
is a classical problem in algebraic geometry. The general $X$ is conjectured to be
not rational but not a single example of non rational cubic fourfold is known.

Cubic fourfolds containing a quartic scroll or a quintic del Pezzo surface are rational (see \cite{F}, \cite{Mo}). Idem for those fourfolds containing a plane and a Veronese surface (see \cite{Tr}).
Beauville and Donagi showed in \cite{BD} that also pfaffian cubic fourfolds are rational.

The closure of the locus of pfaffian cubic fourfolds is a divisor $\mathcal C_{14}$ in the moduli space $\mathcal C $
of all cubic fourfolds, while the fourfolds containing a plane form a divisor $\mathcal C _{8}\,$ (see \cite{H2}). The general fourfold containing a plane is also expected to be non rational.
Nevertheless, Hassett showed in \cite{H1} that there exists a countable infinite collection of divisors in $\mathcal C_8$ which parameterize rational cubic fourfolds. The fourfolds containing a plane are birational
to the total space of a quadric surface bundle by projecting from the plane: Hassett's examples are rational since the associated quadric bundle has a rational section. We call these hypersurfaces {\it trivially rational}. 

Auel et al. (see \cite{ABBV}) have described a divisor in $\mathcal C_8$ whose very general member parameterizes rational but not trivially rational cubic fourfolds. They are all pfaffian, so rational.
In a recent paper, Bolognesi and Russo proved that every cubic hypersurface belonging to $\mathcal C_{14}$ is rational \cite{BR}.

Using results on the Hodge structure of cubic fourfolds and  $K3$ surfaces, we present a family  of cubic fourfolds containing a plane which are not trivially rational.
We don't know if these fourfolds are rational. The rational example in \cite{ABBV} is in our family.

The paper is organized as follows. In Sections \ref{lat} and \ref{k3rk2} we recall some basic notions on lattices and $K3$ surfaces. We focus on $K3$ surfaces of Picard rank two recalling the fundamental work of Nikulin in \cite{N}. Then in \ref{dprk2} we present the $K3$ surfaces of Picard rank two  which are double covers of the plane ramified over a sextic curve. In \ref{defsd} we construct a family $S_{(b,c)}$ of double planes with Picard rank two.  In Section \ref{4foldsplane} we recall how these surfaces are related to cubic $4-$folds containing a plane. Such a cubic $X$ is birational to a quadric bundle $Y \stackrel{\pi}{\longrightarrow}  \PP ^2$ which, in the general case, ramifies over a smooth sextic curve $C.\,$
The Hodge structure of $X$ is strictly related to the Hodge structure of the $K3$ surface $S$ obtained as a double cover of the plane ramified over $C$ and parameterizing the rulings of the quadrics in the fibration $Y \stackrel{\pi}{\longrightarrow}  \PP ^2 \,$ (see \cite[\textsection 1]{V}). We use the following fact: the lattice $A(X)$ of $2-$cycles modulo numerical equivalence on $X$ has rank three and even discriminant if $S$ has Picard rank two and even N\'eron-Severi discriminant (see \cite[\textsection 1 Proposition 2]{V}). In case of $\rk (A(X))=3$ it is known that the quadric bundle
$Y \stackrel{\pi}{\longrightarrow}  \PP ^2\,$ does not have a rational section 
if and only if the discriminant of $A(X)$ is even (see Proposition \ref{iffrk3}).

We prove that if $X$ is not trivially rational, the discriminant $d(A(X))$ is even, without restrictions on the rank of $A(X)$ (see Proposition \ref{noseceven}).

In \ref{thetacar}  we recover the cubic hypersurfaces associated to the double planes $S_{(b,c)}$  using the additional datum of an odd theta characteristic on the discriminant sextic (see \cite{B,V}).

In Theorem \ref{mainteo} we prove that the fourfolds corresponding to $S_{(b,c)}$  with $d$ even are not trivially rational. The rational example in \cite[Theorem 11]{ABBV} correspond to fourfolds associated to $S_{(2,-1)}.$

\noindent
Theorem \ref{mainteo} gives only a sufficient condition for the existence of not trivially rational $4-$folds: there are
cubic fourfolds containing a plane associated to double planes $S_{(b,c)}$ with $b$ odd which are not trivially rational   (see Proposition \ref{369}).

\section{Lattices}\label{lat} 
A {\em lattice} is a free $\ZZ$-module $L$ of finite rank  with a $\ZZ$-valued symmetric bilinear form $b_{L}(x,y).\,$ A lattice  is called {\em even} if the quadratic form $q_L$ associated to the bilinear form has only even values, {\em odd} otherwise. The {\em discriminant} $d(L)$ of a lattice is the determinant of the matrix of its bilinear form. A lattice is called {\em non-degenerate} if the discriminant is non-zero and {\em unimodular} if the discriminant is $\pm 1 .\,$ If the lattice $L$ is non-degenerate, the pair $(s_+, s_-),\,$ where $s_{\pm}$ denotes the multiplicity of the eigenvalue $\pm 1$ for the quadratic form associated to $L \otimes \RR ,\,$ is called {\em signature} of $L .\,$ Finally, we call $\, s_+ +s_- \,$ the {\em rank} of $L$ and $L$ is said {\em indefinite} if the associate quadratic form has both positive and negative values.

Given a lattice $L, \,$  the lattice $L (m)$ is the $\ZZ$-module  $L $ with bilinear form $b_{L (m)}(x,y)=mb_{L}(x,y).\,$  
An {\em isometry} of lattices is an isomorphism preserving the bilinear form. 
Given a sublattice $L ^{\prime} \subset L ,\,$ the embedding is {\em primitive} if $\displaystyle {L  \over L ^{\prime} }$
is free.

Let $L^*= Hom _{\mathbb Z}(L, \mathbb Z)= \{x \in L \otimes \mathbb Q \ : \ b_L(x,l)\in \mathbb Z, \forall l \in L  \}$ be the {\em dual} of the lattice $L.\,$ There is a natural embedding  $L \hookrightarrow L^*$ given by $l \mapsto b_L (l, -).\,$ There is the following

\begin{lemma}\cite[I,Lemma 2.1.]{BPV}\label{subl}
Let $L$ be a non-degenerate lattice. Then
\begin{enumerate}
\item $[L^* : L] = |d(L)|$
\item $[L:L^{\prime}]^2=\displaystyle{\frac{d(L^{\prime})}{d(L)},}$
where $L ^{\prime} \subset L \,$ 
is a sublattice with $\rk(L ^{\prime})=\rk(L).$
\end{enumerate}
\end{lemma} 

\smallskip
Denote by $L$  a non-degenerate even lattice. The bilinear form $b_L$ induces a $\QQ$-valued bilinear form on $L^*$ and so a finite quadratic form
\[
q_{A_L} : L^*/L \longrightarrow \QQ / 2 \ZZ
\]
called the {\em discriminant form} of $L.\,$ The group $ L^*/L:=A_L $ is the {\em discriminant group} of $L$.

\subsection{Examples.} 
\begin{enumerate}
	\item [i)]
The lattice $\langle n \rangle $ is a free $\ZZ$-module of rank one, $\ZZ  \langle e \rangle , \,$with bilinear form \break $b(e,e)=n$.

\item [ii)] The {\em hyperbolic lattice}  is the even, unimodular, indefinite lattice with $\ZZ$-module $\ZZ \langle  e_1,e_2 \rangle $ and bilinear form given by the matrix $\left( \begin{matrix}0&1 \cr 1&0 \end{matrix}\right)$. We write
\[
U = \left \{ \ZZ ^2 ,\ \begin{pmatrix}
0 & 1 \\
1 & 0
\end{pmatrix}   \right \}.
\]

\item [iii)]The lattice $E_8$ has  $\ZZ^8$ as  $\ZZ$-module and the matrix of the bilinear form is the  Cartan matrix of the root system of $E_8$. It is an even, unimodular and positive definite lattice.

\end{enumerate}

\section{K3 surfaces of rank two}\label{k3rk2}

A $K3$  surface is a
smooth projective surface $S$ with trivial canonical class and
$H^1(S,\mathcal O_S)=0.\,$

It is well known that $H^2(S,\ZZ)$ is an even, unimodular, indefinite lattice, with respect to the intersection form  on $S.\,$
It has rank $22,\,$ signature $(3,19)\,$ and it is isomorphic
to 
\[
\Lambda  := U ^{\oplus 3}\oplus E_8 (-1) ^{\oplus 2}.
\]
The lattice $\Lambda$ will be called the {\em  $K3$ lattice}. 
The Hodge numbers are $(1,20,1),\,$ (see \cite[VIII]{BPV}).
Denote by
\[
NS(S) \cong H^2(S,\ZZ) \cap H^{1,1}(S)
\]
the {\em N\'eron-Severi lattice} of $S,\,$ it is a primitive sublattice of $H^2(S,\ZZ)$. Rational, algebraic and homological equivalence coincide on a $K3$ surface.

The orthogonal
complement $T(S)$ of $NS(S)$ in $H^2(S,\ZZ)$ is the {\em transcendental lattice} of $S.\,$

The \textit{rank}
of $S$,  $\rho (S),$ is the rank of $NS(S).\,$ The Hodge Index Theorem implies that
$NS(S)$ has signature $(1,\rho(S)-1)$ and that $T(S)$ has signature $(2,20-\rho(S)).\,$
Let $l \in NS(S)$ be a class with $l^2 >0.\,$ The {\em primitive cohomology} $H^2(S,\ZZ)^0$ 
is the orthogonal complement of the lattice $<l>.$

Main tools for the study of K3 surfaces are the Torelli Theorem (see \cite{LP} and \cite{PSS}) and the Surjectivity of the Period Map (see \cite{T}). The {\em period} of $S$ is given by $[\omega _S] = \PP(H^{2,0}(S))$ in the
period domain 
\[
\Omega = \{ x \in \PP (\Lambda \otimes \CC) | x \cdot x =0,\, x \cdot \bar {x} >0 \} \subset \PP (\Lambda \otimes \CC).\,
\]
By the Torelli Theorem and the Surjectivity of the Period
Map, an element $\omega$ in the period domain determines the $K3$ surface:  given $\omega \in \Omega$ 
 there exists a  $K3$ surface $S_{\omega}$ (unique up
to isomorphism) with period $\omega$  such that
$H^2(S_{\omega}, \ZZ)$ is isometric to $\Lambda .\,$

Nikulin in \cite{N} made a deep study of lattice theory and integral quadratic forms with applications to the study of $K3$ surfaces.  We recall the following which is crucial for our purposes

\begin{theorem}{\cite[Theorem 1.14.4]{N}  \cite[Corollary 2.9]{M}}\label{unique}
If $\rho(S) \leq 10 ,\,$ then every even lattice $M$ of signature $(1,\rho-1 )$ occurs as the N\'eron-Severi group of some $K3$ surface and the primitive embedding $M \hookrightarrow \Lambda$ is unique.
\end{theorem}

\medskip

\begin{corollary}\label{uniquerk2}All even lattices of rank $2$ and
signature   $(1,1)$   occur as the 
N\'eron-Severi lattice $NS(S)$ of some K3 surface $S$ of rank two
and the primitive embedding $NS(S)\hookrightarrow \Lambda $ is unique.  
Any such lattice has the form
\[
M = \left \{ \ZZ ^2 ,\ \begin{pmatrix}
2a & b \\
b & 2c
\end{pmatrix}   \right \}
\]
with $a\geq 0$ and $b^2-4ac >0.$
\end{corollary}

\smallskip

\subsection{$K3$ surfaces double planes of rank two}\label{dprk2}

A double covering of the projective plane $\varphi : S \longrightarrow \PP ^2$ branched along a smooth sextic $C$ is a $K3$ surface: $\varphi _* (\mathcal O _S) \cong \mathcal O_{\mathbb P^2} \oplus \mathcal O(3),\,$ so $H^1(S,\mathcal O _S)=0\,$ and $\omega _S \cong \varphi ^*(\omega_{\mathbb P^2}\otimes \mathcal O(3))\cong 
\mathcal O_S.\,$  
The $K3$ surface $S$ in this case is called a {\em double plane}. For general references on double planes, see \cite{En}  and \cite{S}.
An ample class $l \in NS(S)$ with $l^2=2$ is the pull-back of the class of a line in $\PP ^2$. 
If $S$ has rank two the N\'eron-Severi lattice has the form
\[
L_{(b,c)}  = \left \{ \ZZ ^2 ,\ \begin{pmatrix}
2 & b \\
b & 2c
\end{pmatrix}   \right \}.
\]

\subsubsection{Examples.} 
\begin{enumerate}
\item[i)] Consider $S$ a $K3$ surface double plane ramified over a smooth sextic with
N\'eron-Severi lattice of the form
\[
L_{(1,-1)} = \left \{ \ZZ ^2 ,\ \begin{pmatrix}
2 & 1 \\
1 & -2
\end{pmatrix}   \right \}.
\]
This can be realized by taking a double cover of the plane ramified over a sextic curve having a tritangent line $l.\,$ The pull-back of $l$ to $S$ is a divisor splitting into two irreducible components $l_1,\,l_2$ .  The corresponding divisor classes are linearly independent. Both curves are isomorphic to $l$ and $l_1 ^2 = l_2 ^2 =-2.$

\item[ii)] Analogously,  if the N\'eron-Severi lattice has the form
\[
L_{(2,-1)} = \left \{ \ZZ ^2 ,\ \begin{pmatrix}
2 & 2 \\
2 & -2
\end{pmatrix}   \right \}
\]
the corresponding double plane $S$ can be realized with a ramification sextic $C$ which is tangent to a conic $D$ in $6$ points with multiplicity two. As before, $\varphi ^*(D) = D_1 + D_2,\,$ with $D_1,\,D_2$ isomorphic to $D$ and $D_1 ^2 = D_2 ^2 =-2.$
\end{enumerate}

\smallskip
The previous examples can be generalized as follows. 
\begin{lemma}\label{defsd}
If $b > 0$ and $b^2-4c >0,$ then the lattice 
\[
L_{(b,c)} = \left \{ \ZZ ^2 , \begin{pmatrix}
2 & b \\
b & 2c
\end{pmatrix}   \right \}
\]
is the N\'eron-Severi lattice of a double plane $S_{(b,c)}\,$ with a smooth ramification sextic.
\end{lemma}

\begin{proof} 
The lattice $L_{(b,c)}$ is even and it has signature $(1,1).\,$ By Theorem \ref{unique}  and Corollary \ref{uniquerk2}, $L_{(b,c)}$
occurs as the Picard group of a $K3$ surface: denote by $S_{(b,c)}=S_{\alpha}$ the $K3$ surface
 defined by $\alpha \in \Omega $ with  
$\alpha ^{\perp} = L_{(b,c)}$ and, moreover, generic with this property, hence $L_{(b,c)} = NS(S_{(b,c)}).\,$
Let $H,\,A$ be the classes $(1,0)$ and $(0,1)$ in $NS(S_{(b,c)}),$ respectively. 
For each divisor $\Gamma$ with $\Gamma^2=-2$  we have the Picard--Lefschetz reflection $\pi_{\Gamma}$ of $NS(S_{(b,c)})$ defined by $D\mapsto D+(D\Gamma)\Gamma$. If $D'$ is another divisor on $S_{(b,c)},$ then $\pi_\Gamma(D)\pi_\Gamma(D')=DD'$, because $\Gamma^2=-2$.
The cone of big and nef divisors is a fundamental domain for the group generated by the above reflections (see for example \cite [Chapter 8, Corollary 2.11]{Huy1}). In particular, we can find divisors $\Gamma_i$ with $\Gamma_i\Gamma_j=-2\delta_{i,j}$, $i=1,\dots,l$, such that 
$$
{H'}:=H+\sum_{i=1}^l(H\Gamma_i)\Gamma_i
$$
is nef. Let
$$
{A'}:=A+\sum_{i=1}^l(A\Gamma_i)\Gamma_i.
$$
Thus $NS(S_{(b,c)})=<H,A>=<H',A'>.$ Omitting the prime in the superscript we can thus assume that $H$ is nef.

Let $H=F+M$ be its decomposition in the fixed part $F$ and the mobile part $M,\,$ then $M$ is nef too. Observe that $M^2=H^2=2$ 
(see for example \cite[Chapter 2, Remark 3.3.]{Huy1}). Since, moreover, $M$ is without fixed part by definition, it defines a double cover $\varphi : S_{(b,c)} \longrightarrow \PP ^2.\,$
The ramification curve $C$ is smooth since a point $x \in S$ is singular iff $\varphi (x)$ is a singular point of $C$ (see for example \cite[p.8]{S}).
\end{proof}

\section{Cubic $4-$folds containing a plane}\label{4foldsplane}

Let $X$ be a smooth cubic hypersurface in $\PP ^5 (\CC)  .\,$ 
Consider the cohomology group $H^4(X,\ZZ)$ and 
denote with 
\[
A(X)=H^4(X,\ZZ)\cap H^{2,2}(X)
\] 
the lattice of the middle integral cohomology Hodge classes. Those classes are algebraic  since $X$ verifies the integral Hodge conjecture (see \cite{Mu} and \cite{Zu}). 
The {\em transcendental lattice} $T(X)$ is the orthogonal complement of $A(X)$ (with respect to the intersection form on $X$).

From now on $X$ will indicate a cubic hypersurface in $\mathbb P^5$ containing a plane.
Consider the projection from the plane $P$ onto a plane in $\mathbb P^5$ disjoint from $ P.\,$ Blowing up $X$ along $ P$ one obtains a quadric bundle $\pi \ : \ Y \longrightarrow \mathbb P^2$ branched over $C,\,$ the discriminant sextic. If  $X$ does not contain a second plane intersecting $ P,\,$ the curve $C$ is smooth and this means that the quadrics of the bundle have rank $\geq 3\,$ (see \cite[\textsection 1 Lemme 2]{V}). 

Denote by $Q$ the class of such a quadric. One has $P+Q=H^2,\,$ where $H$ is the hyperplane class associated to the embedding $X \hookrightarrow \PP ^5(\CC) .\,$ The hypersurface  $X$ is said to be {\em very general} if $A(X)\ = \ <H^2,P> \ (= \ <H^2,Q>).\,$ Denote $L:=<H^2,P>^{\perp}.\,$

$X$ is rational  iff  $ Y $ is rational and a sufficient condition for the rationality of $Y$ is the existence of a rational section.

\begin{definition}
We call a cubic hypersurface $X \subset \mathbb P^5$ containing a plane {\it trivially rational} if the associated quadric bundle has a rational section.
\end{definition}

This fact may be translated in a condition on the parity of the intersection of some $2-$cycles on $X.\,$ More precisely, for a  $2-$cycle $T$ in $X$ consider the intersection index
\[
\delta(T)= T \cdot Q  .
\]
Note that $\delta(P) =-2$ and $\delta (H^2) = 2\,$  So, if $X$ is very general the index $\delta$ takes only even values.
There is the following result (see \cite[Theorem 3.1.]{H2}, \cite[Proposition 2]{ABBV}, \cite[Lemma 4.4.]{H1}).
\begin{theorem}\label{has}
A cubic fourfold $X$ containing a plane is trivially rational if and only if there exists a cycle $T$  in $A(X)$ with $\delta(T)\,$ odd. 
\end{theorem}

Using this Theorem it is easy to give (lattice-theoretic) hints to construct cubic fourfolds with $rk(A(X))>2$ 
and not trivially rational (see \cite[Lemma 4.4.]{H1} and \cite[Proposition 2]{ABBV}).

\smallskip

\begin{proposition}\label{iffrk3}
Let $X$ be a cubic fourfold containing a plane with $\rk(A(X))=3.\,$  
Thus $X$ is trivially rational if and only if $d(A(X))$ is odd.
\end{proposition}

\begin{proof}
The quadric bundle $\pi \ : \ Y \longrightarrow \mathbb P^2$ has a rational section if and only if there exists a cycle $T  \in A(X)$ such that  $\delta (T)$ is odd (by Theorem \ref{has}).  Since $A(X)$ has
rank $3$, the sublattice $< H^2, Q, T >$ has finite index, hence Lemma \ref{subl} implies that, if $< H^2, Q, T >$ has odd discriminant, then $d(A(X))$  is odd as well.
\end{proof}

\medskip
Our aim now is to build some geometric examples. To do this, we need to better understand the links between Hodge theory and the geometry on a cubic $4-$fold containing a plane.
Here we follow Voisin \cite[\textsection 1]{V}. 

\smallskip
Let $ \varphi :\ S \longrightarrow \mathbb {P}^2$ be the double cover branched over $C ,\,$ the discriminant sextic of the quadric bundle $Y \longrightarrow \PP ^2.\,$ The surface $S$ parameterizes the rulings of the quadrics of the fibration. Let $F$ be the Fano variety of lines in $X,\,$ the subvariety of the Grassmannian $Gr(1,5)$ parameterizing lines contained in $X.\,$ The divisor $D\subset F$ consisting of lines
meeting $ P \,$ is identified with

\medskip

$
D = \{ (l,s) \in  F \times S \ : \ l \; \mbox{is in the ruling of the quadric  parameterized by} \; \varphi(s) \}.
$

\medskip
giving a $\PP ^1-$bundle 
\begin{equation}\label{p1bundle}
f : D \longrightarrow S.
\end{equation}
The incidence graph restricted to $D$  
\[
D \times X \supset \xymatrix{
   Z_D \ar[d]^q \ar[r]^p & D\\
 X         &     }
\]
defines the Abel-Jacobi map:
\[
\alpha _D = p_* q^*  : H^4(X, \QQ) \longrightarrow  H^2(D,\QQ) 
\]
which induces an isomorphism of Hodge structures, see \cite[\textsection 1 Proposition 1]{V}. Before stating the next result, we recall that we denote by $L$ the orthogonal complement of the lattice $<H^2,P>$ in $H^4(X,\ZZ),$ where $H$ is the hyperplane class and $P$ is the class of a plane contained in $X.$ 

\begin{proposition}(\cite[ \textsection 1 Proposition 2]{V},\cite[Proposition 1]{ABBV})\label{rankdis}
Let $X$ be a smooth cubic fourfold containing a plane. Then $\alpha _D(L)\subset f^*(H^2(S,\ZZ)_0(-1)) $ is a polarized Hodge substructure of index $2$. Moreover, $\alpha _D(T(X)) \subset f^* T(S)(-1)$ is a sublattice of index
$\epsilon$ dividing $2.\,$  In particular, $\rk\, A(X) = \rk\, (NS(S)) + 1$
and $ \ d(A(X)) =(-1)^{\rho(S)-1} 2^{2(\epsilon-1)} d(NS(S))$.
\end{proposition}

We can also derive the following result, which amplifies Proposition \ref{iffrk3}.

\begin{proposition}\label{noseceven}
Let $X$ be a cubic fourfold containing a plane. If $X$ is not trivially rational, then $\alpha _D(T(X)) \subset f^* T(S)(-1)$ is a sublattice of index $2$ and $d(A(X))$ is even.
\end{proposition}

\begin{proof}
The $\mathbb P^1$-bundle $f : D \longrightarrow S$ in (\ref{p1bundle}) produces an element of order two in the Brauer group $Br(S)$ of $S.\,$  The quadric bundle associated to $X$ does not have a rational section if and only if this element is not trivial in $Br(S)$ (see \cite[Proposition 4.7.]{Ku}).
Recall that, if $S$ is a $K3$ surface, then 
\[
Br(S)\cong T(S)^* \otimes \mathbb Q /  \mathbb Z \cong \Hom(T(S),\mathbb Q /  \mathbb Z)
\]
(see for example \cite[\textsection 2.1.]{vG}). 
An element of order $2$ in $Br(S)$ defines a surjective homomorphism 
\begin{equation}\label{surjhom}
\alpha : T(S) \longrightarrow \mathbb Z /  2\mathbb Z
\end{equation}
and thus a sublattice $T_{\alpha}$ of index $2$ in $T(S)$.
Voisin \cite[\textsection 1]{V} and van Geemen \cite[\textsection  9]{vG} give a geometric realization for this element $\alpha$  (see also \cite[\textsection 2]{HVV11}). More precisely, there exists $k \in H^2(S,\ZZ)$ such that 
\[
\alpha_D (L) \cong \{v \in H^2(S,\ZZ)^0 \ : \ <v,k>_S \equiv 0 (\mbox {mod} 2) \}
\]
and $k$ induces an element $\varphi$ in $\Hom (H^2(S,\ZZ)^0, \ZZ / 2 \ZZ)$ which restricts to $\alpha$ in $T(S).$ By definition, $\ker \varphi \cong \alpha_D (L)$ and, since $\alpha_D (T(X))\subseteq \alpha_D (L),$ we have $\alpha_D (T(X))\subseteq f^* (T_{\alpha})(-1).$ Thus $\alpha _D(T(X)) \subset f^* T(S)(-1)$ is a sublattice of index $2$ and $d(A(X))$ is even by Proposition \ref{rankdis}. 
\end{proof}

\begin{remark}
The lattice $T_{\alpha}$ is isometric to the transcendental lattice $T(S,\alpha)$ of the $\alpha$-twisted Hodge structure of $S$ (see  \cite[Proposition 4.7]{Huy3} and \cite[Lemma 2.15]{Huy2}).
If $u,v \in L$ one has that $<u,v>_X=-<\alpha_D(u), \alpha_D(v)>_S$ (see \cite[Proposition 2 ii)]{V}).
Thus Proposition \ref{noseceven} implies that, if $X$ is not trivially rational, then $\alpha_D (T(X))$ is isometric to $T(S,\alpha)(-1)$. 
\end{remark}

\smallskip
\subsection{Theta-characteristics on the ramification curve $C$}\label{thetacar}
A theta-characteristic on a smooth curve $C$ is a line bundle $\kappa$ such
that $\kappa ^{\otimes 2}= K_C .\,$ We write $h^0(\kappa) := \dim H^0(C,\kappa).$

\smallskip
Denote with $Q_x$ a quadric of the bundle $Y \longrightarrow \PP ^2 .\,$ The map $x \mapsto Q_x \cap P$ gives a net of conics whose discriminant curve is a plane cubic $C_1 .\,$ The curve $C_1$ cuts a divisor $2D$  on the sextic $C$ and thus it determines an effective theta-characteristic on $C\,$ (see \cite[\textsection 1 Lemme 7]{V}). Conversely, the cubic hypersurface $X$ is determined by the curve $C$ plus an odd theta-characteristic  (see \cite[\textsection 1 Proposition 4]{V}).
The same result is implied by the following

\begin{proposition}(\cite[Proposition 4.2.]{B})\label{bdet}
Let $C$ be a smooth plane curve of degree $d$, defined by an equation $F =0$ and $\kappa$ an odd theta-characteristic on $C\,$ with $h^0(\kappa)=1.\,$ Thus, $\kappa$ admits a minimal resolution

\smallskip
\[
0 \longrightarrow \mathcal O_{\PP ^2}(-2)^{d-3} \oplus \mathcal O_{\PP ^2}(-3)\stackrel{M}{\longrightarrow} \mathcal O_{\PP ^2}(-1)^{d-3} \oplus \mathcal O_{\PP ^2} \longrightarrow \kappa \longrightarrow 0
\]
with a symmetric matrix $M\in M_{(d-2)\times (d-2)}(\CC[X_0,X_1,X_2])$ satisfying $\det M = F ,\,$ and of the form
\begin{equation}
M = \begin{pmatrix}\label{detm}
L_{1,1} & \dots & L_{1,d-3} & Q_1 \\
 \vdots &  & \vdots & \vdots \\
L_{1,d-3} & \dots & L_{d-3,d-3} & Q_{d-3} \\
Q_1 & \dots & Q_{d-3} & H
\end{pmatrix}
\end{equation}
where the forms $L_{i,j},\, Q_i, \, H$ are linear, quadratic and cubic respectively.

Conversely, the cokernel of a symmetric matrix $M$ as above is an odd theta-characteristic $\kappa$ on $C$ with $h^0(\kappa)=1.\,$
\end{proposition}

\smallskip
We can now prove our main result.

\begin{theorem}\label{mainteo}
Consider the couple $(S_{(b,c)},\kappa)$ where $S_{(b,c)}$ is a double plane defined as in Lemma \ref{defsd} and  $\kappa$ is a theta characteristic on the ramification curve $C$ with $h^0(\kappa)=1.\,$ If $b$ is even, then $(S_{(b,c)},\kappa)$ determines a cubic fourfold which is not trivially rational.
\end{theorem}

\begin{proof}
Let $C$ be the ramification curve of $S:=S_{(b,c)}\,$ and take a theta characteristic $\kappa$ on $C$ with $h^0(\kappa)=1.\,$ 
Proposition \ref{bdet} says that the curve $C$ has a determinantal representation $F= \det M =0$ with
\[
M = \begin{pmatrix}
L_{1,1} & L_{1,2} & L_{1,3} & Q_1 \\
L_{1,2}& L_{2,2} & L_{2,3} & Q_2 \\
L_{1,3} & L_{2,3} & L_{3,3} & Q_3 \\
Q_1 & Q_2 & Q_3 & H
\end{pmatrix}. \ 
\]
The geometric interpretation is the following. Choose projective coordinates $[Z_1,Z_2,Z_3,X_0,X_1,X_2]$ in $\PP ^5(\CC)$ and define the cubic fourfold $X=X(S,\kappa)$ as the zero set 
\[
\sum _{i,j=1} ^3  Z_i Z_j L_{i,j}(X_0,X_1,X_2) + \sum _{i=1} ^3 2  Z_i Q_i(X_0,X_1,X_2)  +H(X_0,X_1,X_2) =0.
\]
The cubic $X$ is smooth and it contains the plane $P := \{X_0=X_1=X_2=0\}. \,$
The curve $C$ is the discriminant of the quadric bundle obtained by projecting the hypersurface $X$ from  $P.$ 

\smallskip
The $K3$ surface $S$ has rank two and $b$ is even, so the discriminant of $NS(S)$ is even. 
This means that $A(X)$ has rank three and even discriminant by Proposition \ref{rankdis}.
That $X$ is not trivially rational follows now from Proposition \ref{iffrk3}. 
\end{proof}

\begin{remark}
Auel et al. in \cite{ABBV} (see Theorem 11) show an explicit example of a pfaffian (hence rational) cubic fourfold associated to a $K3$ surface of type $S_{(2,-1)}.\,$ 
\end{remark}

Theorem \ref{mainteo} gives only a sufficient condition for the existence of not trivially rational $4-$folds. 

\begin{proposition}\label{369}
There exist double planes $S_{(b,c)}$ with $b$ odd determining cubic fourfolds containing a plane which are not trivially rational.
\end{proposition}

\begin{proof}
In \cite[Theorem 4]{ABBV} it is proved that the general fourfold $X$ in one of the irreducible components of $\mathcal C_8 \cap \mathcal C_{14}$ has $A(X)$ with intersection matrix given by  
\begin{equation}
\begin{array}{c|ccc}
& H^2 & P & T\\
\hline 
H^2 & 3 & 1 & 4\\
P & 1 & 3 & 2 \\
T & 4 & 2 & 10 
\end{array} 
\end{equation}
The discriminant sextic $C$ of the quadric bundle associated to $X$ is smooth and let $S=S_{(b,c)}$ the double plane branched on $C.\,$ Since $d(A(X)=36$,  $X$ is not trivially rational by Proposition \ref{iffrk3}. Thus, $d(NS(S)))=-9$ by Proposition \ref{rankdis} and  Proposition \ref{noseceven}. We conclude that $b$ is odd .
\end{proof}

\begin{remark}
Actually, the cubic in the previous example is already known to be rational since it is a pfaffian.
\end{remark}
\bigskip

\centerline{Acknowledgement}
The author would like to thank  Bert van Geemen, Edoardo Sernesi and Michele Bolognesi for helpful remarks and useful discussions. The author warmly thanks the referee for the valuable comments which helped to improve the manuscript and for pointing out some mistakes to correct.
The author is supported by the framework PRIN 2010/11 "Geometria delle Variet\`a
Algebriche", cofinanced by MIUR. Member of GNSAGA.

\bibliographystyle{amsplain}

\bigskip

\end{document}